\documentclass[a4paper,10pt]{article}
\usepackage[T1]{fontenc}
\usepackage[cp1250]{inputenc}
\usepackage[english]{babel}
\usepackage{amsmath}
\usepackage{amsfonts}
\usepackage{amssymb}
\usepackage{mathrsfs}
\usepackage{amsthm}
\usepackage{graphicx}

\usepackage{color}
\usepackage{euscript}
\usepackage{cite}
\usepackage{mathtools}
\usepackage{enumerate}

\newcommand{\eps}{\varepsilon}

\DeclareMathOperator{\I}{I}

\renewcommand{\leq}{\leqslant}
\renewcommand{\geq}{\geqslant}

\newcommand{\R}{\mathbb{R}}

\DeclareMathOperator{\Lip}{Lip}

\newcommand{\scalprod}[2]{\langle{#1},{#2}\rangle}
\newcommand{\eq}[1]{\begin{equation}{#1}\end{equation}}
\newcommand{\mlt}[1]{\begin{multline}{#1}\end{multline}}
\newcommand{\alg}[1]{\begin{align}{#1}\end{align}}

\newcommand{\set}[2]{\{{#1}\mid{#2}\}}
\newcommand{\Set}[2]{\Big\{{#1}\,\Big|\;{#2}\Big\}}
\newcommand{\fdot}{\,\cdot\,}

\newtheorem{Le}{Lemma}

\newtheorem{St}[Le]{Proposition}
\newtheorem{Th}{Theorem}
\newtheorem{Cor}[Le]{Corollary}
\newtheorem{Rem}[Le]{Remark}

\begin{document}

\title{On Fourier multipliers with rapidly oscillating symbols}

\author{Dmitriy Stolyarov\thanks{Supported by the Russian Science Foundation grant 19-71-30002.}}

\maketitle


\begin{abstract}
We provide asymptotically sharp bounds for the~$L_p$ norms of the Fourier multipliers with symbols~$e^{i\lambda \varphi(\xi/|\xi|)}$, where~$\lambda\in \R$ is a large parameter.
\end{abstract}

\section{Setting}
The motivation for writing this note comes from Maz'ja's Problem $4.2$ in~\cite{Mazya2018}. We provide an imprecise citation.

\emph{'Consider the singular integral operator~$A_\lambda$ with the symbol~$\partial B_1 \ni \omega \mapsto \exp(i\lambda \varphi(\omega))$,
where~$\varphi$ is a smooth real-valued function on~$\partial B_1$, and~$\lambda$ is a large real parameter. Find the sharp value of the parameter~$\kappa$ such that the estimate 
\eq{
\|A_\lambda\|_{L_p(\R^d)\to L_p(\R^d)} \leq c|\lambda|^{\kappa|\frac12-\frac{1}{p}|},
}
holds true; here~$1 < p < \infty$, $c$ depends on~$d, p$, and the function $\varphi$.'} 

This estimate is clearly true when~$d=1$ and~$\kappa=0$. We will solve the problem for larger~$d$ by proving the following results\footnote{The paper~\cite{Mazya2018} suggested~$\kappa = d-1$. It might be the case that the value~$d-1$ instead of~$d$ appeared in~\cite{Mazya2018} occasionally, because the leading term of the estimate in~\cite{Maz'jaHaikin1975} was~$|\lambda|^{d|1/2-1/p|}$.}.
\begin{St}\label{Example}
Let~$d\geq 2$. There exists a smooth~$\varphi$ defined on the unit sphere in~$\R^d$ and  a tiny constant~$C$ such that
\eq{\label{CounterExampleIneq}
\|A_\lambda\|_{L_p\to L_p} \geq C|\lambda|^{d|\frac12-\frac{1}{p}|},
}
provided~$|\lambda|$ is sufficiently large. 
\end{St}
The case~$d=2$ of Proposition~\ref{Example} had been implicitly considered in~\cite{DPV2006} (see Theorem~$6$ of that paper). We also prove that the inequality reverse to~\eqref{CounterExampleIneq} always holds true. 
\begin{Th}\label{Inequality}
Let~$d$,~$p$, and~$\varphi$ be fixed. There exists a constant~$c$ such that
\eq{
\|A_\lambda\|_{L_p(\R^d)\to L_p(\R^d)} \leq c|\lambda|^{d|\frac12-\frac{1}{p}|}.
}
\end{Th} 
A similar bound
\eq{
\|A_\lambda\|_{L_p(\R^d)\to L_p(\R^d)} \leq c|\lambda|^{d|\frac12-\frac{1}{p}|}\big(\log |\lambda|\big)^{|2/p-1|},\qquad |\lambda| \geq 2,
}
was proved in~\cite{Maz'jaHaikin1975}, where such type bounds were utilized to obtain sufficient conditions for the~$L_p$-continuity of the operators~$f\mapsto \int a(\cdot,\xi)\hat{f}(\xi)\,d\xi$. We will prove Proposition~\ref{Example} in Sections~\ref{S2} and~\ref{S2,5} by an accurate computation that resembles the stationary phase method; the first of the two sections deals with the simpler case~$d=2$ and the other covers the more involved case of larger~$d$ (the presentation in Section~\ref{S2,5} is slightly less detailed). Though the computations in the case~$d \geq 3$ are longer, they follow the same route as in the~$d=2$ case. The proof of Theorem~\ref{Inequality} is a combination of the sharp multiplier theorem in the spirit of Mikhlin and Hormander from~\cite{Seeger1988} and soft interpolation techniques, see the details in Section~\ref{S3}.

Some results of this paper were independently obtained in~\cite{BuljKovac2022}; namely, Theorem~\ref{Inequality} and the cases of even~$d$ in Proposition~\ref{Example} are considered there. We note that~\cite{BuljKovac2022} suggests a more fundamental treatment of the whole circle of problems similar to Maz'ja's problem discussed here.  I wish to thank Vladimir Maz'ja for attracting my attention to the problem and fruitful discussions, and to Vjekoslav Kova\v{c} for providing me with the references~\cite{BuljKovac2022} and~\cite{DPV2006}. 

\section{Example for~$d=2$}\label{S2}
Let~$\Phi$ be a non-negative smooth function of a single variable supported in~$[1/2,3/2]$ and non-zero on~$[2/3,4/3]$. Let~$\chi$ be another smooth function supported in~$[-1/2,1/2]$, non-zero on~$[-1/3,1/3]$, and not exceeding one everywhere. Finally, assume that~$\varphi\colon \R\to \R$ is a~$2\pi$-periodic smooth function whose value coincides with~$2\pi \theta$ when~$\theta \in [-1,1]$. Consider the function~$M_\lambda \colon \R^2\to \R$,
\eq{
M_\lambda(\xi,\eta) = e^{i\lambda \varphi(\theta)}\chi(\theta)\frac{\Phi(\rho)}{\rho},\quad \hbox{where}\quad \xi = \rho\cos \theta, \eta = \rho\sin\theta,\quad \theta \in (-\pi,\pi),\ \hbox{and}\ \rho > 0,
}
(we have the polar change of variables in mind). We reserve the names~$x,y$ for the variables on the 'real' Fourier side. We also set
\eq{
x = r\sin\theta_{x,y}\quad \hbox{and}\quad y = r\cos\theta_{x,y},\quad \hbox{where}\ \theta_{x,y}\in (-\pi,\pi)\ \text{and}\ r > 0,
}
for the dual polar coordinates (we interchange the cosine and sine functions for convenience of further computations).
\begin{Le}\label{AsympFormulad=2}
The inequality
\eq{\label{AsymptoticFormula}
\Big|r\hat{M}_\lambda(x,y) - e^{-2\pi i \lambda \theta_{x,y}}\Phi(\lambda/r)\chi(-\theta_{x,y})\Big| \leq 0.1
}
holds true provided~$r\in [3\lambda/4,3\lambda/2]$,~$\theta_{x,y}\in [-1/10, 1/10]$, and~$\lambda$ is sufficiently large.
\end{Le}
\begin{proof}
We write down the definition of the Fourier transform and perform the polar change of variables:
\mlt{
\hat{M}_\lambda(x,y) = \iint\limits_{\R^2}M_\lambda(\xi,\eta)e^{-2\pi i(x\xi + y\eta)}\,d\xi\,d\eta = \int\limits_{-\frac12}^\frac12\int\limits_{\frac12}^{\frac32}e^{i\lambda \varphi(\theta)}\chi(\theta)\Phi(\rho)e^{-2\pi i (x\rho \cos \theta + y \rho \sin \theta)}\,d\rho\,d\theta =\\
\int\limits_{-\frac12}^\frac12\hat{\Phi}(x \cos \theta + y \sin \theta) e^{i\lambda \varphi(\theta)}\chi(\theta)\,d\theta = 
\int\limits_\R\hat{\Phi}(r \sin(\theta + \theta_{x,y})) e^{i\lambda \varphi(\theta)}\chi(\theta)\,d\theta =\\
\int\limits_\R\hat{\Phi}(r \sin \theta) e^{i\lambda \varphi(\theta-\theta_{x,y})}\chi(\theta-\theta_{x,y})\,d\theta = \int\limits_\R\hat{\Phi}(r \sin\theta) e^{2\pi i\lambda(\theta-\theta_{x,y})}\chi(\theta-\theta_{x,y})\,d\theta,
}
by our requirements on the function~$\varphi$. We multiply this identity by~$r$ and make yet another change of variable:
\eq{\label{OneDimensionalDilation}
r\hat{M}_\lambda(x,y) = e^{-2\pi i \lambda \theta_{x,y}}\int\limits_{\R}\hat{\Phi}\Big(r \sin \frac{s}{r}\Big) e^{2\pi i\frac{\lambda}{r} s}\chi\Big(\frac{s}{r}-\theta_{x,y}\Big)\,ds.
}
We note that since~$\chi$ is supported in~$[-1/2,1/2]$ and~$|\theta_{x,y}|\leq 1/10$, the integrand is non-zero only when~$|s| \leq r$. In particular,~$|r \sin s/r| \geq s/10$. Let~$R$ be a fixed large number such that
\eq{
|\hat{\Phi}(x)| \leq |x|^{-2},\quad \hbox{when}\quad |x| > R, \qquad \hbox{and} \qquad \int\limits_{|x| > R}\frac{dx}{x^2} < 10^{-5}.
}
Then, since we assume~$\|\chi\|_{L_\infty} \leq 1$,
\eq{\label{TailEstimate}
\Big|\int\limits_{|s| > R}\hat{\Phi}\Big(r \sin \frac{s}{r}\Big) e^{2\pi i\frac{\lambda}{r} s}\chi\Big(\frac{s}{r}-\theta_{x,y}\Big)\,ds\Big| \leq \int\limits_{|s| > R}\frac{100}{s^2}\,ds < \frac{1}{100}.
}
Thus, we need to show that
\eq{
\int\limits_{|s| < R}\hat{\Phi}\Big(r \sin \frac{s}{r}\Big) e^{2\pi i\frac{\lambda}{r} s}\chi\Big(\frac{s}{r}-\theta_{x,y}\Big)\,ds
}
is close to~$\Phi(\lambda/r)\chi(-\theta_{x,y})$. For that, we wish to replace~$r \sin \frac{s}{r}$ by~$s$. Note that~$|s| < R$ and~$r$ is large. Therefore,
\eq{
\Big|\hat{\Phi}\Big(r \sin \frac{s}{r}\Big) - \hat{\Phi}(s)\Big| \leq \|\hat{\Phi}\|_{\Lip}\, r\frac{|s|^3}{r^3} = \|\hat{\Phi}\|_{\Lip} \frac{|s|^3}{r^2}.
}
What is more,
\eq{
r^{-2}\Big|\int\limits_{|s| < R} |s|^3\chi\Big(\frac{s}{r}-\theta_{x,y}\Big)\,ds\Big|=O(|\lambda|^{-2})\to 0,
}
when~$\lambda \to \infty$ and~$r\asymp \lambda$. Thus, it suffices to prove
\eq{
\int\limits_{|s| < R}\hat{\Phi}(s) e^{2\pi i\frac{\lambda}{r} s}\chi\Big(\frac{s}{r}-\theta_{x,y}\Big)\,ds
}
is close to~$\Phi(\lambda/r)\chi(-\theta_{x,y})$. Using the same estimate as in~\eqref{TailEstimate}, we see that this quantity is~$1/100$-close to
\eq{
\int\limits_{\R}\hat{\Phi}(s) e^{2\pi i\frac{\lambda}{r} s}\chi\Big(\frac{s}{r}-\theta_{x,y}\Big)\,ds,
}
which converges to the desired value.
\end{proof}
\begin{Rem}
The convergence is uniform with respect to~$\theta_{x,y}\in [-1/10,1/10]$ and~$r/\lambda\in [3/4,3/2]$.
\end{Rem}
\begin{Rem}
One may prove the asymptotic formula
\eq{
r\hat{M}_\lambda(x,y) \to e^{-2\pi i \lambda \theta_{x,y}}\Phi(\lambda/r)\chi(-\theta_{x,y})
}
by choosing~$R$ slowly tending to infinity as~$\lambda \to \infty$. 
\end{Rem}
\begin{Cor}\label{LpCor}
Let us denote the Fourier multiplier with the symbol~$M_\lambda$ by the same symbol. Then, there exists a tiny positive constant~$c$ such that
\eq{
\|M_\lambda\|_{L_p\to L_p} \geq c\lambda^{|\frac{2}{p}-1|},\quad p\in [1,\infty], 
}
provided~$\lambda$ is sufficiently large.
\end{Cor}
\begin{proof}
By duality, it suffices to consider the case~$p \leq 2$. Let us start with the endpoint~$p=1$: we wish to show that~$\|K_\lambda\|_{L_1} \geq c\lambda$, where~$K_\lambda$ is short for~$\hat{M}_{\lambda}$. We assume~$\lambda$ is so large that~\eqref{AsymptoticFormula} holds true. Then, by our assumptions about the functions~$\Phi$ and~$\chi$, 
\alg{
&\Big|\hat{M}_\lambda(x,y)\Big| \gtrsim \lambda^{-1},\quad (x,y)\in \mathcal{R}_\lambda,\qquad \text{where}\\
&\mathcal{R}_\lambda = \Set{(x,y)\in \R^2}{\sqrt{x^2+y^2} \in [4\lambda/5,5\lambda/4]\ \hbox{and}\ |x|/|y| < \frac{1}{100}}.
}
The notation~$A\gtrsim B$ means~$A\geq C B$ for a uniform constant~$C$ (with respect to~$\lambda$). The area of~$\mathcal{R}_\lambda$ is bounded away from zero by~$c\lambda^2$, which leads to the desired estimate.

The estimate for the~$L_p\to L_p$ norm is a little bit trickier. Let~$f$ be the characteristic function of a small ball centered at the origin. We will show that 
\eq{\label{ConvEst}
\Big|K_\lambda*f(x,y)\Big| \gtrsim \lambda^{-1},\qquad \hbox{when}\ (x,y)\in\mathcal{R}_\lambda,
}
the parameter~$\lambda$ is sufficiently large, and the radius of the ball is sufficiently small ($1/100$ suffices). We use~\eqref{AsymptoticFormula}:
\mlt{
\Big|K_\lambda*f(x,y)\Big| = \Big|\!\!\!\int\limits_{(u,z)\in B_{0.01}(x,y)}\!\!\!\!\!\! K_\lambda(u,z)\,du\,dz\Big|\geq\\ \Big|\!\!\!\int\limits_{(u,z)\in B_{0.01}(x,y)}\!\!\!\!\! e^{-2\pi i \lambda \theta}r^{-1}\Phi(\lambda/r)\chi(-\theta)\,du\,dz\Big| - \frac{1}{1000\lambda},
}
here~$r = \sqrt{u^2 + z^2}$ and~$\theta = \theta_{u,z}$. The absolute value of the latter integral may be rewritten as
\eq{
\Big|\!\!\!\int\limits_{(u,z)\in B_{0.01}(x,y)}\!\!\!\!\! e^{-2\pi i \lambda (\theta - \theta_{x,y})}r^{-1}\Phi(\lambda/r)\chi(-\theta)\,du\,dz\Big|.
}
Since~$|\lambda(\theta-\theta_{x,y})| < 1/10$ on the domain of integration (since the function~$(u,z)\mapsto \lambda \theta_{u,z}$ is~$10$-Lipschitz there), the absolute value of the integral is comparable to~$(x^2 + y^2)^{-1/2}\Phi(\lambda/\sqrt{x^2 + y^2})\chi(-\theta_{x,y})$ and~\eqref{ConvEst} is proved.

By~\eqref{ConvEst}, we have
\eq{
\|K_\lambda*f\|_{L_p} \gtrsim \lambda^{-1} |\mathcal R_\lambda|^{1/p} = \lambda^{2/p-1},
}
which finishes the proof.
\end{proof}
\begin{proof}[Proof of Proposition~\ref{Example}]
The function~$(\xi,\eta)\mapsto \chi(\theta)\Phi(\rho)/\rho$ is smooth and compactly supported, denote the Fourier multiplier with this symbol by~$T$. Then,~$M_\lambda = TA_{\lambda}$ and~$\|M_{\lambda}\| \leq \|T\|\|A_{\lambda}\|\lesssim \|A_\lambda\|$, where the norms are the~$L_p\to L_p$ norms of the operators.
\end{proof}

\section{Example for~$d\geq 3$}\label{S2,5}
Consider the function
\eq{
M_\lambda(\xi) = e^{2\pi i \lambda \varphi(\xi/|\xi|)}\tilde{\Phi}(|\xi|)\tilde{\chi}(\xi/|\xi|),\qquad \xi \in\R^d.
}
The function~$\tilde{\Phi}\colon \R\to \R$ is a smooth non-negative function supported inside~$[1/2,3/2]$. The functions~$\varphi$ and~$\tilde{\chi}$ are defined on the unit sphere. We suppose~$\tilde{\chi}$ is smooth, non-negative, and supported in a tiny neighborhood~$U_\eps$ of~$(1,0,\ldots,0)$. As for the function~$\varphi$, we provide an explicit formula:
\eq{
\varphi(\zeta) = |\zeta_1 - 1|^2 + \sum\limits_{j=2}^d\zeta_j^2 = 2-2\zeta_1,\qquad \zeta \in S^{d-1}.
}
The main feature we will use is that this linear function possesses certain 'curvature' when restricted to the unit sphere. In particular, if we restrict the function~$\varphi$ to the intersection of~$S^{d-1}$ with a linear hyperplane passing through~$U_\eps$,~$\varphi$ attains the minimal value only at the point that has maximal first coordinate on this intersection. 

We wish to compute the value~$\hat{M}_\lambda(x)$.  More specifically, we will prove the following analog of Lemma~\ref{AsympFormulad=2}. It yields the case~$d\geq 3$ in Proposition~\ref{Example} in the same way as Lemma~\ref{AsympFormulad=2} yields Corollary~\ref{LpCor} and the case~$d=2$ of Proposition~\ref{Example}.
\begin{St}
There exists a parallelepiped~$R_\lambda$ with dimensions~$\gtrsim\lambda$ and two functions~$L_\lambda \colon R_\lambda \to \R$ and~$A_\lambda \colon R_\lambda \to \R_+$ such that
\eq{
\Big|\hat{M}_\lambda(x) - e^{2\pi i L_\lambda(x)}A_{\lambda}(x)\Big| \lesssim \lambda^{-\frac{d}{2}-1},
}
the function~$L_\lambda$ is~$O(1)$-Lipschitz, and~$A_\lambda(x) \gtrsim \lambda^{-d/2}$ when~$x\in R_\lambda$, provided~$\lambda$ is sufficiently large.
\end{St}
\begin{proof}
Let~$r = |x|$ and assume~$x/|x|$ is close to the vector~$(0,1,0,0,\ldots,0)$ (the specific vector is chosen for notational convenience, the only important thing is that it is orthogonal to~$(1,0,0\ldots,0)$). Later we will see that~$x/|x|$ should be close, but not equal to~$(0,1,0,\ldots,0)$. We will need to use another spherical coordinate system, we will call it 'new' (its choice depends on~$x$), and the one we have worked in before is called 'stationary' (the stationary system is simply the usual Euclidean coordinate system, not a spherical one). The coordinates in the new system will be~$(\rho,\tilde{\theta})$, where~$\rho \in \R_+$ and~$\tilde{\theta} \in S^{d-1}$. In the new coordinate system,~$x$ has coordinates~$(r,0,1,0,\ldots,0)$. Let~$z_x$ be the point on the intersection of the linear hyperplane~$x^\perp$ with the unit sphere at which~$\varphi$ attains its minimum (on this intersection). One may see~$z_x$ is close to~$(1,0,0,\ldots,0)$ (in the stationary system). We choose the new system in such a way that~$z_x$ has coordinates~$(1,1,0,0,\ldots,0)$. This information completely defines the new coordinate system in the case~$d=3$; in higher dimensions we make some choice of possible new coordinate systems. We go further and parametrize~$\tilde{\theta}$ with the points on the tangent plane at the point~$z_x$:
\eq{
\tilde{\theta} = \Big(\cos\theta_2\cos\theta_3\ldots \cos \theta_d,\, \sin\theta_2,\, \cos\theta_2\sin\theta_3,\,\ldots,\, \cos\theta_2\cos\theta_3\ldots \sin \theta_d\Big).
}
Here~$\theta = (\theta_2,\theta_3,\ldots,\theta_d)\in \R^{d-1}$ lies inside a small ball~$B_\eps(0)$. We have
\eq{
\hat{M}_\lambda(x) = \int\limits_{B_\eps(0)}\int\limits_{\frac12}^{\frac32}e^{2\pi i (-r\rho \sin\theta_2 + \lambda\varphi_x(\theta))}\tilde{\Phi}(\rho)\tilde{\chi}(\tilde{\theta}) J(\rho,\theta)\,d\rho\,d\theta.
}
The function~$J(\rho,\theta) = \rho^{d-1}\tilde{J}(\theta)$ comes from the  spherical change of variables. Note that~$\tilde{J}(\theta)\ne 0$ on~$B_\eps(0)$. Let us specify the function~$\varphi_x\colon B_\eps \to \R$. Let~$(1,0,0,\ldots,0)$ in the stationary coordinates correspond  to~$(1,\alpha)$,~$\alpha \in S^{d-1}$, in the new coordinates. Then, with the notation~$\scalprod{\cdot}{\cdot}$ for the Euclidean scalar product,
\mlt{\label{Explicitvalues}
\varphi_x(\theta) = 2 - 2\scalprod{\tilde{\theta}}{\alpha} =\\ 2-2\alpha_1\cos\theta_2\cos\theta_3\ldots \cos \theta_d -2\alpha_2 \sin\theta_2 - 2\alpha_3 \cos\theta_2\sin\theta_3 - \ldots - 2\alpha_d\cos\theta_2\cos\theta_3\ldots \sin \theta_d.
}
We have chosen our coordinate system in such a fashion that~$\varphi_x$, when restricted to the set
\eq{
x^\perp\cap S^{d-1} = \set{\tilde{\theta}}{\tilde{\theta}_2 = 0}
} 
(which is the same as~$\theta_2 = 0$), attains its minimum at the origin. Thus,~$\alpha_3=\alpha_4=\ldots=\alpha_d = 0$. We may compute~$\alpha_2 = \scalprod{x/|x|}{(1,0,0\ldots,0)} = x_1/|x|$. Recall we also have~$\alpha_1^2 + \alpha_2^2 = 1$, so this equation completely defines~$\alpha$. Thus,
\eq{
\varphi_x(\theta) = 2 - 2\alpha_1\cos\theta_2\cos\theta_3\ldots \cos \theta_d - 2\alpha_2 \sin\theta_2,\qquad \alpha_2 = \frac{x_1}{|x|}, \alpha_1 = \frac{\sqrt{|x|^2 - x_1^2}}{|x|}.
}
Later we will require~$\alpha_2$ to be non-zero. 

Let us write~$\psi = \varphi_x$,~$\Phi(\rho) = \rho^{d-1}\tilde{\Phi}(\rho)$, and~$\chi(\theta) = \tilde{\chi}(\tilde{\theta})\tilde{J}(\theta)$ for brevity. Note that the function~$\chi$ depends on~$x/|x|$ (we will only use that since~$\chi$ is non-zero on a small neighborhood of~$(1,0,0,\ldots,0)$ in the stationary system,~$\tilde{\chi}(0)\ne 0$, provided~$x/|x|$ is sufficiently close to~$(0,1,0,\ldots,0)$). We wish to compute
\eq{\label{eq33}
\int\limits_{B_\eps(0)}\int\limits_{\frac12}^{\frac32}e^{2\pi i (-r\rho \sin\theta_2 + \lambda\psi(\theta))}\Phi(\rho)\chi(\theta)\,d\rho\,d\theta = \int\limits_{B_\eps(0)}e^{2\pi i \lambda\psi(\theta)} \hat{\Phi}(r\sin\theta_2)\chi(\theta)\,d\theta.
}
Let us call the latter integral simply~$\I$. We will be using the notation~$\theta = (\theta_2,\theta_{[2]})$, where~$\theta_{[2]} = (\theta_3,\theta_4,\ldots,\theta_d)\in \R^{d-2}$. We pick a large natural number~$N$ and write Taylor's expansion
\eq{
\psi(\theta) = \psi_0(\theta_{[2]}) + \theta_2\psi_1(\theta_{[2]}) + \sum\limits_{j=2}^{N}\theta_2^j\psi_j(\theta_{[2]}) + R_N^\psi(\theta),
}
where~$R_N^\psi(\theta) = O(|\theta_2|^{N+1})$ as~$\theta_2 \to 0$, the functions~$\psi_j$ are smooth, and~$\theta \in B_\eps(0)$. In particular,
\eq{
\psi_0(\theta_{[2]}) = \psi(0,\theta_{[2]})\quad \text{and}\quad \psi_1(\theta_{[2]}) = \frac{\partial \psi}{\partial \theta_2}(0,\theta_{[2]}),
}
which, in view of~\eqref{eq33}, turns into
\alg{\label{SpecificValues}
&\psi_0(\theta_{[2]}) = -2\alpha_1\cos\theta_3\cos\theta_4\ldots\cos\theta_d;\\
&\psi_1(\theta_{[2]}) = -2\alpha_2,
}
here~$\theta_{[2]}\in B_\delta(0)$, where~$\delta$ is a small number.
We perform dilation with respect to~$\theta_2$ (similar to~\eqref{OneDimensionalDilation}):
\eq{
r\I = \int\limits_{\R^{d-1}}\exp\bigg(2\pi i \lambda\Big(\sum\limits_{j=0}^{N}(\theta_2/r)^j\psi_j(\theta_{[2]}) + R_N^\psi(\theta_2/r,\theta_{[2]})\Big)\bigg) \hat{\Phi}\big(r\sin (\theta_2/r)\big)\chi\big(\theta_2/r,\theta_{[2]}\big)\,d\theta.
}
Note that since~$\hat{\Phi}$ is a Schwartz function, we may 'cut the Schwartz tail' similar to~\eqref{TailEstimate}:
\mlt{\label{35}
r\I =\\ \int\limits_{|\theta_2|\leq r^{\kappa}}\exp\bigg(2\pi i \lambda\Big(\sum\limits_{j=0}^{N-1}(\theta_2/r)^j\psi_j(\theta_{[2]}) + R_N^\psi(\theta_2/r,\theta_{[2]})\Big)\bigg) \hat{\Phi}\big(r\sin (\theta_2/r)\big)\chi\big(\theta_2/r,\theta_{[2]}\big)\,d\theta + O(r^{-d}),
}
here~$\kappa$ is a fixed small positive number to be specified later. Let us assume~$N$ is odd. We write the expansion
\eq{
\hat{\Phi}\Big(r\sin \frac{\theta_2}{r}\Big) = \hat{\Phi}\Big(\theta_2 + \sum\limits_{\genfrac{}{}{0pt}{-2}{j\geq 2}{2 \mid j}}^{N-1} \frac{\theta_2^{j+1}}{(j+1)!r^j} + rR(\theta_2/r)\big)\Big), 
}
where~$R(t) = O(|t|^{N+1})$ when~$t\to 0$ and~$R$ is a smooth function. We require~$\kappa (N+1)\leq 1/10$ (we choose~$\kappa$ sufficiently small), and, thus, may assume~$rR(\theta_2/r) = O(r^{-N+\frac12})$ uniformly on the domain~$|\theta_2|\leq r^{\kappa}$. Expanding further (with Taylor's formula for the function~$\hat{\Phi}$), we arrive at 
\eq{\label{expansionPhi}
\hat{\Phi}\Big(r\sin \frac{\theta_2}{r}\Big) = \hat{\Phi}(\theta_2) + \sum\limits_{j=1}^{N-1}r^{-j}\Psi_j(\theta_2) + R_N^\Phi(\theta_2,r).
}
Here~$R_N^\Psi(\theta_2,r) = O(r^{-N+\frac12})$ uniformly when~$|\theta_2|\leq r^\kappa$, and, moreover, the~$\Psi_j$ are Schwartz functions. What is more,
\eq{\label{expansionChi}
\chi(\theta_2/r,\theta_{[2]}) = \chi(0,\theta_{[2]}) + \sum\limits_{j=1}^{N-1}\theta_2^j r^{-j}\chi_{j}(\theta_{[2]}) + R_N^{\chi}(\theta),
} 
where~$R_N^{\chi}(\theta) = O((\theta_2/r)^{N})$; in particular
\eq{
R_N^{\chi}(\theta) = O\big(r^{-N+\frac12}\big) \qquad \text{when}\qquad |\theta_2| \leq r^\kappa.
} 
We may write yet another asymptotic formula
\mlt{\label{expansionExp}
\exp\Big(2\pi i \lambda\Big(\sum\limits_{j=2}^{N}\theta_2^j r^{-j}\psi_j(\theta_{[2]}) + R_N^\psi(\theta_2/r,\theta_{[2]})\Big)\Big) - 1=\\
\sum\limits_{k=1}^{N-1} \frac{(2\pi i \lambda)^k}{k!}\Big(\sum\limits_{j=2}^{N}\theta_2^j r^{-j}\psi_j(\theta_{[2]}) + R_N^\psi(\theta_2/r,\theta_{[2]})\Big)^k + O\Big(\Big(\frac{|\lambda||\theta_2|^2}{r^2}\Big)^N\Big) = \\ \sum\limits_{k=1}^{N-1}\sum\limits_{j= 2k}^{N} \lambda^kr^{-j}\tilde{\psi}_{k,j}(\theta) + O(r^{-N+\frac12}),
}
since we assume~$\lambda \asymp r$,~$|\theta_2|\leq r^\kappa$ and~$\kappa(N+1)\leq 1/10$. The functions~$\tilde{\psi}_{k,j}(\theta)$ are of the form~$p(\theta_2)q(\theta_{[2]})$, where both factors are smooth.
Plugging~\eqref{expansionPhi},~\eqref{expansionChi}, and~\eqref{expansionExp} into~\eqref{35}, we obtain
\mlt{\label{FinalExpansion}
r\I = \int\limits_{|\theta_2|\leq r^{\kappa}}e^{2\pi i (\lambda\psi_0(\theta_{[2]})+ \frac{\lambda}{r}\theta_2\psi_1(\theta_{[2]}))}\hat{\Phi}(\theta_2)\chi(0,\theta_{[2]})\,d\theta +\\ \sum\limits_{k=1}^{N-1}\sum\limits_{j=2k}^{N}\int\limits_{|\theta_2|\leq r^{\kappa}}\lambda^k r^{-j} e^{2\pi i (\lambda\psi_0(\theta_{[2]})+ \frac{\lambda}{r}\theta_2\psi_1(\theta_{[2]}))}\Theta_{k,j}(\theta)\,d\theta + O(r^{-N+\frac12+\kappa}),
}
here~$\Theta_{k,j}$ are some smooth functions of the form
\eq{
\Theta_{k,j} =\sum\limits_{\alpha} a_{\alpha}(\theta_2)b_{\alpha}(\theta_{[2]}),
}
where the~$a_{\alpha}$ are Schwartz functions of a single variable, the~$b_{\alpha}$ are smooth and compactly supported, and the index~$\alpha$ runs through a finite set. Recall~$\lambda/r\asymp 1$  and let~$\lambda,r\to \infty$. We claim that the terms with the indices~$k,j$ in the expansion~\eqref{FinalExpansion} behave as~$\lambda^{-\frac{d-2}{2} - j+k}$. From now on we set~$N = 2d$,~$\kappa = 1/(30d)$ and forget about the remainder term. Let us compute the asymptotically sharp value for the leading term of~\eqref{FinalExpansion} and then say that the computation for lower order terms is similar. Here the computation is:
\mlt{
\int\limits_{|\theta_2|\leq r^{\kappa}}e^{2\pi i \lambda(\psi_0(\theta_{[2]})+ \frac{\lambda}{r}\theta_2\psi_1(\theta_{[2]}))}\hat{\Phi}(\theta_2)\chi(0,\theta_{[2]})\,d\theta =\\ 
\int\limits_{\R^{d-1}}e^{2\pi i (\lambda\psi_0(\theta_{[2]})+ \frac{\lambda}{r}\theta_2\psi_1(\theta_{[2]}))}\hat{\Phi}(\theta_2)\chi(0,\theta_{[2]})\,d\theta + O(r^{-d}) = \\
\int\limits_{\R^{d-2}}e^{2\pi i \lambda\psi_0(\theta_{[2]})}\Phi\Big(\frac{\lambda}{r}\psi_1(\theta_{[2]})\Big)\chi(0,\theta_{[2]})\,d\theta_{[2]} + O(r^{-d}).
}
Recall that we have adjusted our new coordinates in the way that~$\psi_0$ attains its minimal value at the origin only. One may see~$\psi_0$ is smooth and, thus, falls under the scope of the standard stationary phase method (see, e.g., p. 344 in~\cite{Stein1993}):
\mlt{\label{StatPhase}
\int\limits_{\R^{d-2}}e^{2\pi i \lambda\psi_0(\theta_{[2]})}\Phi\Big(\frac{\lambda}{r}\psi_1(\theta_{[2]})\Big)\chi(0,\theta_{[2]})\,d\theta_{[2]} =\\ \frac{c_0}{\sqrt{H}} \lambda^{-\frac{d-2}{2}}e^{2\pi i \lambda \psi_0(0)}\Phi\Big(\frac{\lambda}{r}\psi_1(0)\Big)\chi(0) + O(\lambda^{-\frac{d}{2}}),\qquad \lambda \to \infty,
}
where~$c_0$ is an absolute constant and~$H$ is the determinant of the Hesse matrix~$\frac{\partial^2 \psi_0}{\partial \theta_{[2]}^2}(0)$ (one may observe from~\eqref{SpecificValues} that this value does not vanish). 

Pick a tiny number~$\nu > 0$ and consider a small spherical cap close to the vector~$(0,1,0\ldots,0)$ (in the stationary system) on which~$x_1\in (-1.1\nu|x|,-0.9\nu|x|)$. In particular, this cap does not contain the vector~$(0,1,0\ldots,0)$ itself. The number~$\nu$ is chosen in such a way that~$\tilde{\chi}(z_x) > 0$ when~$x$ belongs to the said spherical cap. Then, we have
\eq{
\hat{M}_\lambda(x) = r^{-1}\lambda^{-\frac{d-2}{d}} e^{2\pi i \lambda\frac{\sqrt{|x|^2 - x_1^2}}{|x|}}\Phi\Big(-\frac{2\lambda x_1}{r|x_1|}\Big) c_0(x) + O(\lambda^{-\frac{d+2}{2}}),
} 
where~$c_0(x)$ is a smooth non-zero function (it equals a constant times a positive function). Indeed, the leading term in~\eqref{FinalExpansion} is evaluated with the help of~\eqref{StatPhase}, whereas the lower order terms do not contribute stronger than~$O(\lambda^{-\frac{d+2}{2}})$ by a similar formula.

We adjust~$r$ in such a way that~$\Phi\Big(-\frac{2\lambda x_1}{r|x_1|}\Big)$ is non-zero (in other words,~$r\in [c_1\lambda,c_2\lambda]$, where~$c_1$ and~$c_2$ are absolute constants). Recall that~$x/|x|$ belongs to a fixed spherical cap. This means  there exists a parallelepiped~$R_\lambda$ of dimensions~$\gtrsim \lambda$ such that
\eq{
\hat{M}_{\lambda}(x)=e^{2\pi i \lambda L_\lambda(x)}A_\lambda(x),\quad x\in R_\lambda,
}
where~$L_\lambda$ is a~$100$-Lipschitz function and~$A_\lambda \gtrsim \lambda^{-d/2}$.
\end{proof} 

\section{Estimate}\label{S3}
The proof of Theorem~\ref{Inequality} relies upon the following sharp version of the Mikhlin--Hormander multiplier theorem. In this theorem,~$\Phi$ is an auxiliary smooth radial non-zero function that is compactly supported in~$\R^d\setminus \{0\}$.
\begin{Th}[Theorem $1$ in~\cite{Seeger1988}]
Let~$m$ be a function in~$\R^d$, let~$M$ be the Fourier mulitplier with the symbol~$m$. Then,
\eq{
\|M f\|_{L_{1,r}}\lesssim \Big(\sup\limits_{t > 0}\|\Phi(\cdot) m(t\fdot)\|_{B_2^{d/2,1}}\Big)\|f\|_{\mathcal{H}_1},
} 
provided~$r > 2$.
\end{Th}
The spaces~$\mathcal{H}_1$,~$L_{1,r}$, and~$B_2^{d/2,1}$ are the real Hardy class, the Lorentz, and the Besov spaces. The interpolation formulas (see Theorem~$5.3.1$ in~\cite{BerghLofstrom1976} and~\cite{FRS1974}, correspondingly)
\eq{
L_p = (L_{1,r},L_2)_{2-2/p,p}\quad \hbox{and}\quad L_p = (\mathcal{H}_1,L_2)_{2-2/p,p},\qquad p\in (1,2),
}
lead to the estimate
\eq{
\|M\|_{L_{p}\to L_p}\lesssim \|M\|_{\mathcal{H}_1\to L_{1,r}}^{2/p-1} \|M\|_{L_2\to L_2}^{2-2/p} \lesssim \Big(\sup\limits_{t > 0}\|\Phi(\cdot) m(t\fdot)\|_{B_2^{d/2,1}}\Big)^{2/p-1}
}
for any~$p\in(1,2)$  (the constant depends on~$p$), provided~$m$ is uniformly bounded. Since the symbol of~$A_\lambda$ is homogeneous of order zero, Theorem~\ref{Inequality} follows from the multiplier theorem above and the lemma below.
\begin{Le}
Let~$\Phi$ be a smooth function compactly supported in~$\R^d\setminus \{0\}$. Then,
\eq{
\Big\|\Phi(\xi)e^{i\lambda \varphi(\xi/|\xi|)}\Big\|_{B_2^{d/2,1}}\lesssim |\lambda|^{d/2},\qquad |\lambda| \geq 1.
}
\end{Le}
\begin{proof}
Let
\eq{
M_\lambda(\xi) = \Phi(\xi)e^{i\lambda \varphi(\xi/|\xi|)},\qquad \xi \in \R^d.
}
The estimtate
\eq{\label{LinftyEstimate}
\|\partial^\alpha M_\lambda\|_{L_{\infty}} \lesssim |\lambda|^{|\alpha|},
}
where~$\alpha$ is an arbitrary multiindex, follows from direct differentiation. Since~$\Phi$ is compactly supported,~\eqref{LinftyEstimate} yields
\eq{
\|M_\lambda\|_{W_2^{d}} \lesssim |\lambda|^{d},\quad |\lambda| \geq 1.
}
Therefore, it suffices to prove the multiplicative bound
\eq{\label{BesovInterp}
\|G\|_{B_2^{d/2,1}}\lesssim \|G\|_{L_2}^\frac12\|G\|_{W_2^d}^\frac12.
}
A slightly simpler estimate
\eq{
\|G\|_{W_2^{d/2}}\lesssim \|G\|_{L_2}^\frac12\|G\|_{W_2^d}^\frac12
}
follows from the Cauchy--Schwarz inequality. The inequality~\eqref{BesovInterp} is a little bit more tricky. It follows from the interpolation formula
\eq{
(L_2, W_2^{d})_{\frac12,1} = B_2^{d/2,1}
}
(see~\cite{BerghLofstrom1976}, Theorem~$6.2.4$). We provide an elementary proof of~\eqref{BesovInterp} in the appendix.
\end{proof}

\section{Appendix}
Let~$B_r(x)$ be the Euclidean ball centered at~$x\in \R^d$ and of radius~$r>0$. Let
\eq{
a_k = \|\hat{G}\|_{L_2(B_{2^{k}}(0)\setminus B_{2^{k-1}}(0))},\quad k \geq 1,
}
and~$a_0 = \|\hat{G}\|_{L_2(B_1(0))}$.
Then,
\eq{
\|G\|_{L_2} = \Big(\sum\limits_{k\geq 0}a_k^2\Big)^\frac12,\quad \|G\|_{W_2^d}\asymp \Big(\sum\limits_{k\geq 0}2^{2dk}a_k^2\Big)^\frac12,\quad \hbox{and}\quad \|G\|_{B_2^{d/2,1}} = \sum\limits_{k\geq 0} 2^{dk/2}a_k.
}
With the notation~$A = 2^{d/2}$,~\eqref{BesovInterp} is rewritten as
\eq{\label{RaisedFourth}
\sum\limits_{k \geq 0} A^{k}a_k \lesssim \Big(\sum\limits_{k\geq 0}a_k^2\Big)^\frac14 \Big(\sum\limits_{k\geq 0}A^{4k}a_k^2\Big)^\frac14.
}
We raise the expression on the left hand side to the fourth power, use the AM-GM inequality together with summation of geometric series:
\mlt{
\sum\limits_{p,q,r,s \geq 0}A^{p+q+r+s}a_pa_qa_ra_s = 24\sum\limits_{p\geq q\geq r \geq s}A^{p+q+r+s}a_p a_q a_r a_s\leq \\
12\sum\limits_{p\geq q\geq r\geq s} \Big(A^{-p+3q+r+s}a_q^2a_s^2 + A^{3p-q+r+s} a_p^2a_r^2\Big)\leq \\
12\sum\limits_{q,s}a_q^2a_s^2A^{3q+s}\sum\limits_{\genfrac{}{}{0pt}{-2}{p\geq q}{r\leq q}}A^{-p+r} + 12\sum\limits_{p,r}a_p^2a_r^2A^{3p+r}\sum\limits_{\genfrac{}{}{0pt}{-2}{q\geq r}{s\leq r}}A^{-q+s} \lesssim_A\\
\Big(\sum\limits_{q,s}a_q^2a_s^2A^{3q+s}+\sum\limits_{p,r}a_p^2a_r^2A^{3p+r}\Big)\lesssim\\ \sum\limits_{k,l}a_k^2a_l^2A^{4l} =  \Big(\sum\limits_{k\geq 0}a_k^2\Big) \Big(\sum\limits_{k\geq 0}A^{4k}a_k^2\Big),
}
and arrive at the fourth power of the right hand side of~\eqref{RaisedFourth}.
\bibliography{mybib}{}
\bibliographystyle{amsplain}

Dmitriy Stolyarov

St. Petersburg State University, 14th line 29B, Vasilyevsky Island, St. Petersburg, Russia.

d.m.stolyarov@spbu.ru.

\end{document}